\documentclass{amsart}
\usepackage{amsthm,amssymb,amsmath,enumerate,graphicx, tikz}
\usepackage{amscd}
\usepackage{setspace}
\usepackage{comment}
\usepackage{hyperref}
\newtheorem{theorem}{Theorem}[section]

\newtheorem* {theorem*}{Theorem}

\newtheorem{lemma}[theorem]{Lemma}

\newtheorem{corollary}[theorem]{Corollary}
\newtheorem{claim}[theorem]{Claim}  
\newtheorem{conjecture}[theorem]{Conjecture}
\theoremstyle{definition}
\newtheorem{definition}[theorem]{Definition}
\newtheorem{notation}[theorem]{Notation}

\theoremstyle{remark}

\newcommand{\D}{\mathcal{D}}
\newcommand{\C}{\mathcal{C}}
\newcommand{\M}{\mathcal{M}}
\newcommand{\cS}{P}
\newcommand{\cT}{Q}
\newcommand{\cs}{\alpha}
\newcommand{\ct}{\beta}

\newcommand{\refT}[1]{Theorem~\ref{#1}}

\newcommand{\refCl}[1]{Claim~\ref{#1}}
\newcommand{\refL}[1]{Lemma~\ref{#1}}

\newcommand{\refCon}[1]{Conjecture~\ref{#1}}

\usepackage{graphicx} 

\let\OLDthebibliography\thebibliography
\renewcommand\thebibliography[1]{
  \OLDthebibliography{#1}
  \setlength{\parskip}{0pt}
  \setlength{\itemsep}{0pt plus 0.3ex}
}

\title{The list chromatic number of the intersection of two generalized partition matroids}
\author{He Guo}\thanks{Faculty of Mathematics, Technion, Haifa 32000, Israel. Email: {\tt hguo@campus.technion.ac.il}}
\date{May 2024}

\begin{document}

\begin{abstract}
A famous theorem of Galvin states that the list chromatic number of the intersection of two partition matroids equals its chromatic number. Kir\'aly and B\'{e}rczi et. al. conjectured that this equality holds for any two matroids. We prove this conjecture and a conjecture by Aharoni--Berger for any two generalized partition matroids. 
\end{abstract}

\maketitle

\section{Introduction}

An (\emph{abstract}) \emph{complex}~$\C$ is a finite set of finite sets which is closed under taking subsets, i.e., if $S\in\C$, then $T\in\C$ for every~$T\subseteq S$. 
 The set~$V=\bigcup \C$ is called the \emph{ground set} of ~$\C$.
Each~$S\in\C$ is called a \emph{face} of~$\C$. Given~$U\subseteq V$, $\C[U]:=\{S\subseteq U\mid S\in \C\}$ is called the \emph{subcomplex of~$\C$ induced on~$U$}.

A complex~$\M$ is a \emph{matroid} if~$\emptyset\in\M$, and for any $S, T \in \M$ with $|S|<|T|$, there exists $v\in T\setminus S$ such that $S\cup\{v\}\in\M$. A set $S\in\M$ is also called \emph{independent} in~$\M$ (a term taken from the matroid of linearly independent sets in $\mathbb{F}_p^n$, it should not be confused with independent sets of vertices in a graph). For $A\subseteq V$, the set \emph{spanned} by~$A$ in~$\M$ is
\[ span_\M(A):=A\cup\{x\in V\setminus A: \{x\}\cup I\not\in \M \text{ for some $I\in\M[A]$} \}. \]

The \emph{expansion number} of~$\M$ is
\[ \Delta(\M):=\max_{\emptyset\neq S\subseteq V}\frac{|span_\M(S)|}{|S|}.  \]

In a \emph{generalized partition matroid}~$\M$ the ground set is partitioned  into sets~$P_1,\dots,P_a$ and positive integers $p_1,\dots,p_a$ are given, such that $I\in \M$ if and only if $|I\cap P_i|\le p_i$ for each $1\le i\le a$. We call~$p_1,\dots, p_a$ the \emph{constraints} of the parts~$P_1,\dots, P_a$, respectively. 
If  $p_1=\cdots=p_a=1$, the matroid is plainly called a \emph{partition matroid}.

For a complex~$\C$, a \emph{$\C$-respecting coloring} is a set of faces in~$\C$ whose union is~$V(\C)$.
    The \emph{chromatic number} $\chi(\C)$ of~$\C$ is the minimum number of faces in a~$\C$-respecting coloring. For example, when~$\C=\mathcal{I}(G)$,  the collection of independent sets in a graph~$G$,~$\chi(\C)$ is the (classic) chromatic number of~$G$. When~$\C=\mathcal{M}(G)$, the collection of all matchings in the graph~$G$,~$\chi(\C)$ is the edge chromatic number of~$G$.

Extending a theorem of Nash-Williams~\cite{williams} for graph arboricity, Edmonds~\cite{edmonds1965minimum} 
proved the following result.
\begin{theorem}\label{thm:williams}
    For a matroid~$\M$, 
    \[ \chi(\M)= \lceil \Delta(\M)\rceil. \]
\end{theorem}

K\"onig's edge coloring theorem says that if  $\M_1$ and $\M_2$ are partition matroids on the same ground set, then  
\begin{equation*}
 \chi(\M_1\cap\M_2) = \max(\chi(\M_1), \chi(\M_2))=\max(\Delta(\M_1),\Delta(\M_2)).
\end{equation*}

In~\cite{ABR}, Aharoni and Berger offered the following conjecture (see also~\cite{AB}).
\begin{conjecture}\label{conj:AB}
    For any two matroids  $\M_1$ and $\M_2$ on the same ground set,
    \[ \chi(\M_1\cap\M_2)\le\max\big(\chi(\M_1),\chi(\M_2)+1 \big).  \]
\end{conjecture}
In~\cite{AB}, using topology, it is proved that $\chi(\M_1\cap\M_2)\le2\max\big(\chi(\M_1),\chi(\M_2) \big)$. 
By a refined topological method, the author and Berger~\cite{Eli} proved that 
$\chi(\M_1\cap\M_2)\le\chi_\ell(\M_1\cap\M_2)\le \chi(\M_1)+\chi(\M_2)$ (see the definition of~$\chi_\ell$ below).

Given a complex~$\C$ and lists $(L_v:v\in V(\C))$ of {\em permissible colors}, a {\em  list coloring} with respect to these lists is a function $f:V\rightarrow \cup_{v\in V}L_v$ satisfying $f(v)\in L_v$ for every $v\in V$. It is said to be $\C$-{\em respecting} if $f^{-1}(c) \in \C$ for every color $c\in \cup_{v\in V}L_v$.
The \emph{list chromatic number}~$\chi_\ell(\C)$ is the minimal integer~$p$ such that any lists~$(L_v:v\in V)$ satisfying $|L_v| = p$ for each $v\in V$ has a $\C$-respecting list coloring. 

If $L_v=[p]$ for every $v\in V$ then a $\C$-respecting list coloring is just a $\C$-respecting coloring by~$p$ colors. Therefore \[\chi_\ell(\C)\ge \chi(\C).\] 

As is well-known,  $\chi_\ell(\C)/\chi(\C)$ can be arbitrarily large. For example, when~$\C$ is the collection of independent sets in the complete bipartite graph~$K_{n,n}$,  $\chi(\C)=2$ and $\chi_\ell(\C)=\Theta(\log n)$~\cite{erdos1979choosability}. But, as shown by Seymour ~\cite{SEYMOUR1998150}, 
in matroids $\chi_\ell=\chi$. 
In~\cite{ABGK}, answering a question by~Kir\'aly~\cite{kiralyk} and B\'{e}rczi, Schwarcz, and Yamaguchi~\cite{Berczi}, the author together with Aharoni, Berger, and Kotlar,  proved that when $\C$ is the intersection of~$k$ matroids on the same ground set, then 
\[ \chi_\ell(\C)\le k\chi(\C). \]

Settling a conjecture of Dinitz, Galvin~\cite{GALVIN1995153} proved a strengthening of  K\"onig's edge coloring theorem: 
\begin{theorem}\label{thm:galvin}
     For any two partition matroids~$\M_1$ and~$\M_2$ on the same ground set, 
    \[\chi_\ell(\M_1\cap\M_2)=\chi(\M_1\cap\M_2).\]
\end{theorem}

Kir\'aly~\cite{kiraly2013open} and B\'{e}rczi, Schwarcz, and Yamaguchi~\cite{Berczi} conjectured that this is in fact true for the intersection of any pair of matroids. 

\begin{conjecture}\label{conj:berczi}
For any two matroids~$\M_1$ and $\M_2$ on the same ground set,
\[ \chi_\ell(\M_1\cap\M_2)=\chi(\M_1\cap\M_2). \]
\end{conjecture}

Kir\'aly and Pap~\cite{kiraly2010list} proved the conjecture when $\M_1 $ and $\M_2$ are both transversal matroids or both are of rank 2; or the ground set is the disjoint union of two arborescences having the same root, $\M_1$ is the graphic matroid and $\M_2$ is the partition matroid with the parts formed by the in-stars.
We prove~\refCon{conj:AB} and~\refCon{conj:berczi} when~$\M_1$ and~$\M_2$ are two generalized partition matroids. 

\begin{theorem}\label{thm:main}
     For any two generalized partition matroids  $\M_1$ and $\M_2$ on the same ground set,  
    \[  \chi_\ell(\M_1\cap\M_2)=\chi(\M_1\cap\M_2)=\max\big(\chi(\M_1),\chi(\M_2)\big).  \]
\end{theorem}

\section{A graph terminology formulation}

A \emph{graph} is a pair $G=(V,E)$, where~$V$ is a finite set and~$E$ is a family of unordered pairs from~$V$. The elements of~$V$ are called the \emph{vertices} and the elements of~$E$ are called \emph{edges}. In the definition of graph we use the term ``family" rather than ``set", to indicate that the same pair of vertices may occur several times in~$E$. A pair occurring more than once in~$E$ is called a \emph{multiple} edge. To emphasize this, sometimes we call a graph \emph{multigraph} if it has multiple edges.

A graph is \emph{bipartite} if its vertex set can be divided into two parts such that neither of the parts contains an edge of the graph.

    Given a bipartite multigraph $G$ and a function $b:V(G)\rightarrow \mathbb{Z}_+$, a \emph{simple $b$-matching} is an edge subset~$F$ of~$G$ such that $deg_F(v):=|\{e\in F\mid v\in e \}|\le b_v$ for every $v\in V(G)$.
When $b\equiv 1$,  ``simple $b$-matching'' is just the familiar notion of ``matching''.

Let $\M_1$ and~$\M_2$ be generalized partition matroids on the same ground set~$U$, where~$\M_1$ has parts $\cS_1,\dots, \cS_a$ with respective constraints~$p_1,\dots, p_a$ and~$\M_2$ has parts~$\cT_1,\dots, \cT_b$ with respective constraints~$q_1,\dots, q_b$. We construct a bipartite graph~$G$ whose two sides of are $\{\cS_1,\dots,\cS_a\}$ and $\{\cT_1,\dots, \cT_b\}$, and each edge~$e_u=\{\cS_i, \cT_j\}$ corresponds to an element $u\in U$ such that $u\in \cS_i\cap \cT_j$. Set $b(\cS_i)=p_i$ for each $1\le i\le a$ and $b(\cT_j)=q_j$ for each $1\le j\le b$. Then a face in~$\M_1\cap\M_2$ corresponds to a simple $b$-matching in~$G$.

For the other way around, given a bipartite multigraph~$G$ with two sides~$X$ and~$Y$, and a function $b:V(G)\rightarrow\mathbb{Z}_+$, we can construct two generalized partition matroids~$\M_1$ and~$\M_2$ on ground set~$E(G)$ such that the parts of~$\M_1$ are $\Gamma(x):=( \{x,y \}  )_{\{x,y\}\in E(G)}$, i.e., the family of all edges incident with~$x$, with constraint $b(x)$ for each $x\in X$, and the parts of~$\M_2$ are $\Gamma(y):=( \{x,y \}  )_{\{x,y\}\in E(G)}$ with constraint~$b(y)$ for each $y\in Y$. Then a simple~$b$-matching in~$G$ is a face in~$\M_1\cap\M_2$, and vice versa. 

Together with~\refL{lemma:Delta} in the next section, some immediate results of~\refT{thm:main} are the following.

\begin{corollary}
For a bipartite multigraph~$G$ and $b:V(G)\rightarrow\mathbb{Z}_+$,
    the minimal number of simple $b$-matchings whose union is $E(G)$ is
    $\lceil \max_{v\in V(G)}\frac{deg_G(v)}{b(v)}\rceil$.
\end{corollary}

Given a bipartite multigraph $G$  and $b:V(G)\rightarrow \mathbb{Z}_+$, the \emph{list $b$-edge chromatic number} is the minimum~$k$ such that for any lists $L=(L_e:e\in E(G))$ with $|L_e|=k$ for each~$e\in E(G)$, there exists a choice function~$f:E(G)\rightarrow \cup_{e\in E(G)}L_e$ such that $f(e)\in L_e$ for each $e\in E(G)$ and $f^{-1}(c)$ is a simple $b$-matching for each $c\in \cup_{e\in E(G)}L_e$.

\begin{corollary}
Given a bipartite multigraph $G$  and $b:V(G)\rightarrow \mathbb{Z}_+$, 
the list $b$-edge chromatic number is
\[\Big\lceil \max_{v\in V(G)}\frac{deg_G(v)}{b(v)}\Big\rceil.\]
\end{corollary}

\section{The chromatic number}

From the definitions there follows:
\begin{lemma}\label{lemma:Delta}
    If $\M$ is a generalized partition matroid  with parts $P_1, \ldots ,P_a$ and constraints $p_1, \ldots,p_a$, then  $\Delta(\M)=\max_{1\le i\le a} \frac{|P_i|}{p_i}$.
\end{lemma}
\begin{proof}
    Suppose $X\subseteq V(\M)$ satisfies 
    \[\frac{|span_{\M}(X)|}{|X|}=\Delta(\M)=\max_{\emptyset\neq S\subseteq V(\M)}\frac{|span_\M(S)|}{|S|}.\]

By induction on the number~$d$, it is easy to prove that for sequences of numbers $(y_\ell)_{1\le \ell\le d}$ and $(x_{\ell})_{1\le \ell \le d}$, if $y_\ell\ge x_\ell>0$ for each $1\le \ell\le d$, then
\[  \frac{\sum_{\ell=1}^d y_\ell}{\sum_{\ell=1}^d x_\ell}\le \max_{1\le \ell \le d}\frac{y_\ell}{x_\ell}. \]
And it is easy to prove that $span_{\M}(X)$ is equal to the disjoint union of~$span_{\M}(X\cap P_i)$ for $1\le i\le a$.
Therefore
\begin{equation}\label{eq:Delta=}
\begin{split}
     \frac{|span_{\M}(X)|}{|X|} &=\frac{\sum_{1\le i\le a: X\cap P_i\neq \emptyset} |span_{\M}(X\cap P_i)|}{\sum_{1\le i\le a:X\cap P_i\neq\emptyset}|X\cap P_i|}\\ 
     &\le \max_{1\le i\le a: X\cap P_i\neq \emptyset}\frac{|span_{\M}(X\cap P_i)|}{|X\cap P_i|}.
\end{split}
\end{equation}
Since for $Y\subseteq P_i$,
\begin{equation*}
    span_{\M}(Y)=\begin{cases}
Y &\text{ if $|Y|<p_i$,}\\
P_i &\text{ if $|Y|\ge p_i$,}        
    \end{cases}
\end{equation*}
the maximum in~\eqref{eq:Delta=} can be attained when $X\subseteq P_i$ and $|X|=p_i$ for some~$i$, which completes the proof.
\end{proof}

We first prove a formula for $\chi(\M_1\cap\M_2)$:
\begin{theorem}\label{thm:chi2GP}
Let $\M_1$ and $\M_2$ be two generalized partition matroids on the same ground set. Then 
    \[ \chi(\M_1\cap\M_2)=\max\big(\chi(\M_1),\chi(\M_2)\big)=\max\Big(\lceil\Delta(\M_1)\rceil,\lceil\Delta(\M_2)\rceil\Big).  \]
\end{theorem}

For the proof we need some definitions and notation.

A \emph{directed graph} (\emph{digraph}) is a pair $D=(V,A)$, where~$V$ is a finite set and~$A$ is a family of ordered pairs from~$V$. The elements of~$V$ are called the \emph{vertices} and the elements of~$A$ are called \emph{directed edges}. The \emph{direction} of~$(u,v)\in E$ is from~$u$ to~$v$. And we abbreviate $(u,v)$ as $uv$ in this note.
Again, the term ``family" is to indicate that the same pair of vertices may occur several times in~$A$. A pair occurring more than once in~$A$ is called a \emph{multiple} directed edge.
\begin{notation}
    Given a digraph $D=(V,A)$ and $U\subseteq V$,  let
\[ \delta^{in}(U):=\{ wu\in A\mid w\in V\setminus U, u\in U   \},\]
\[\delta^{out}(U):=\{ uw\in A\mid u\in U, w\in V\setminus U    \}. \]
\end{notation}

\begin{notation}
    For a function~$f:A\rightarrow\mathbb{R}$ and $B\subseteq A$, let $f(B):=\sum_{a\in B}f(a)$.
\end{notation}

\begin{definition}
A function $f:A\rightarrow\mathbb{R}$ on the edge set of a digraph $D=(V,A)$
is  a \emph{circulation} if $f(\delta^{in}(\{v\}))=f(\delta^{out}(\{v\}))$ for every $v\in V$.
\end{definition}

\begin{theorem}[Hoffman's circulation theorem (see, e.g., Theorem 11.2 in~\cite{SchrijverCO})]\label{thm:hoffman}
    Let $D=(V,A)$ be a digraph (allowing multiple directed edges) and let $d,c: A\rightarrow\mathbb{R}$ be functions satisfying $d\le c$. Then there exists a circulation $f$ satisfying $d\le f\le c$ if and only if
    \begin{equation}\label{eq:hoffman}
        d(\delta^{in}(U))\le c(\delta^{out}(U)) 
    \end{equation}
    for each $U\subseteq V$. Moreover, if $d$ and $c$ are integral, then $f$ can be chosen to be integral.
\end{theorem}


\begin{proof}[Proof of~\refT{thm:chi2GP}]
The second equality in the conclusion of the theorem is by~\refT{thm:williams}. It remains to prove the first one.
   We assume that $\M_1$ has parts $\cS_1,\cdots, \cS_a$ with respective constraints $p_1,\dots,p_a$ and $\M_2$ has parts $\cT_1,\dots,\cT_b$ with respective constraints $q_1,\dots, q_b$. Assume the ground set is~$V$. 
   Therefore $F\subseteq V$ is in~$\M_1\cap \M_2$ if and only if 
   \begin{equation*}
       |F\cap \cS_i|\le p_i\quad \text{and}\quad |F\cap \cT_j|\le q_j
   \end{equation*}
    for every $1\le i\le a$ and $1\le j\le b$. 
   By~\refL{lemma:Delta},
     \[\Delta(\M_1)=\max_{1\le i\le a}\frac{|\cS_i|}{p_i} \text{\quad  and \quad} \Delta(\M_2)=\max_{1\le j\le b}\frac{|\cT_j|}{q_j}.\]
     Let $C=\max\big(\chi(\M_1),\chi(\M_2)\big)=\max\Big(\lceil\Delta(M_1)\rceil,\;\lceil\Delta(\M_2)\rceil\Big)$.

It is easy to see that $\chi(\M_1\cap\M_2)\ge C$, since any $\M_1\cap\M_2$-respecting coloring is an $\M_1$-respecting coloring and an $\M_2$-respecting coloring.

For the other direction, let $N_1,\dots, N_C$ be a set of disjoint faces of $\M_1\cap\M_2$ such that their union includes the maximum number of elements of $V$. If $N_1\cup\cdots\cup N_C=V$, we are done.

Suppose there exists $v\in V\setminus (N_1\cup\cdots\cup N_C)$. Then 
\[v\in \cS_{i}\cap \cT_{j}\] 
for some $1\le i\le a$ and $1\le j\le b$. Since 
$C\cdot p_i\ge\Delta(\M_1)\cdot p_i>|\cS_i\setminus\{v\}|$ and $C\cdot q_j>|\cT_j\setminus\{v\}|$,
there exist $1\le \cs\le C$ and $1\le \ct\le C$ such that 
\begin{equation}\label{eq:NsSiNtTj}
    |N_\cs\cap \cS_i|\le p_i-1\quad\text{and}\quad |N_\ct\cap \cT_j|\le q_j-1.
\end{equation}

If $\cs=\ct$, then $N_\cs\cup\{v\}\in \M_1\cap \M_2$, which together with other $N_k$'s includes one more elements of~$V$ than $N_1\cup\cdots\cup N_C$, a contradiction.
Thus $\cs\neq \ct$. 

Furthermore, 
\begin{equation}\label{eq:SiTjin}
    |N_\ct\cap \cS_i|=p_i\quad \text{and}\quad |N_\cs\cap \cT_j|=q_j,
\end{equation}
otherwise either $N_\ct\cup\{v\}$ or $N_\cs\cup\{v\}$ is in~$\M_1\cap \M_2$, which together with other~$N_k$'s includes one more elements of~$V$ than $N_1\cup\cdots\cup N_C$, a contradiction.

Without loss of generality, we may assume $\cs=1$ and $\ct=2$.
We construct a bipartite multigraph~$G$ as the following. 
Let \[L:=\{\cS_\ell\mid \text{ there exists $u\in N_1\cup N_2$ such that $u\in \cS_\ell$}\},\]
 \[R:=\{\cT_r \mid \text{ there exists $u\in N_1\cup N_2$ such that $u\in \cT_r$} \} \]
be the two parts of the vertex set of~$G$. Especially $\cS_i\in L$ and $\cT_j\in R$ by~\eqref{eq:SiTjin}.
Then let each edge~$e$ of the bipartite graph~$G$ between~$\cS_\ell$ in~$L$ and~$\cT_r$ in~$R$ represent an element $u\in N_1\cup N_2$ such that $u\in \cS_\ell\cap \cT_r$. So the number of edges between~$\cS_\ell$ and~$\cT_r$ is the number of $u\in N_1\cup N_2$ such that $u\in \cS_\ell\cap \cT_r$.
 Furthermore for $k=1,2$, each $\cS_\ell\in L$, and each $\cT_r\in R$,
 we denote 
 \begin{equation}\label{eq:degkST}
     deg_{k}(\cS_\ell):=|N_k\cap \cS_\ell|\quad\text{and}\quad deg_{k}(\cT_r):=|N_k\cap \cT_r|,
 \end{equation}
 and for every $L'\subseteq L$ and $R'\subseteq R$, we denote
 $N_k(L',R')$ for the family of edges of~$G$ that are between~$L'$ and~$R'$ representing an element in~$N_k$.
We have 
\begin{equation}\label{eq:degk}
    deg_{k}(\cS_\ell)\le p_\ell\quad \text{and}\quad deg_{k}(\cT_r)\le q_r
\end{equation}
for every $\cS_\ell\in L$ and $\cT_r\in R$ and $k=1,2$.

Then based on~$G$, we construct a digraph $D=(V(D),A)$ by directing the edges of~$G$ from the $L$-side to the $R$-side, and on each such directed edge~$e$ we set $d(e)=0$ and $c(e)=1$. (To avoid clutter, we use the same notation~$e$ for such edge in~$G$ and in~$D$.) We add two new vertices $s$ and $t$, and direct $s$ to all vertices in $L$ with 
\[d(s\cS_\ell)=deg_1(\cS_\ell)+deg_2(\cS_\ell)-p_\ell \quad \text{and}\quad c(s\cS_\ell)=p_\ell,\] and direct all vertices in $R$ to $t$ with 
\[d(\cT_rt)=deg_1(\cT_r)+deg_2(\cT_r)-q_r\quad \text{and} \quad c(\cT_rt)=q_r.\] 
And we direct $t$ to $s$ with $d(ts)=0$ and $c(ts)=\infty$. Thus $V(D)=V(G)\cup\{s,t\}$.
Then $N_1\cup N_2$ induces a circulation~$f$ satisfying $d\le f\le c$ in the following way.
For each edge $e$ between $L$ and $R$, we set
\[ f(e) = \begin{cases}
    1 & \text{if the element represented by $e$ is in $N_1$,} \\
    0 & \text{if the element represented by $e$ is in $N_2$.}
  \end{cases}
\]
 We set 
 \[f(s\cS_\ell)=deg_1(\cS_\ell),\quad f(\cT_rt)=deg_1(\cT_r),\quad \text{and}\quad f(ts)=|N_1|.\] 
 Then it can be verified that $f$ is a circulation and satisfies $d\le f\le c$ by~\eqref{eq:degk} and~\eqref{eq:degkST}.
Then by~\refT{thm:hoffman}, the condition~\eqref{eq:hoffman} holds for every vertex subset~$U$ of~$D$.

Next based on~$D$, we construct a digraph~$D'$ by adding one more directed edge~$e_v$ from~$\cS_i$ to~$\cT_j$ representing $v$. Thus $V(D')=V(D)$ and $A'=A\cup\{e_v\}$. We set $d(e_v)=0$ and $c(e_v)=1$. We set 
\[d(s\cS_i)=deg_1(\cS_i)+deg_2(\cS_i)+1-p_i,\]
\[d(\cT_jt)=deg_1(\cT_j)+deg_2(\cT_j)+1-q_j,\] and $c,d$ on other edges are same as those of $A$.
Note that in $A'$, by~\eqref{eq:degk} and $deg_1(\cS_i)\le p_i-1$ via assumption~\eqref{eq:NsSiNtTj} and~\eqref{eq:degkST}, we still have 
\[d(s\cS_i)=deg_1(\cS_i)+deg_2(\cS_i)+1-p_i\le p_i= c(s\cS_i).\] 
Similarly $d(\cT_jt)\le c(\cT_jt)$. Therefore $d\le c$ for the integral~$d$ and~$c$.

If we can check that for such $D',d,c$, the condition~\eqref{eq:hoffman} holds for every $U\subseteq V(D')$, then there exists an integral circulation~$g$ satisfying $d\le g\le c$. Especially, on the edges between $L$ and $R$, $g$ is zero-one. Then let $M$ be the set of elements represented by the edges with $g$-value 1, and $N$ be those of $g$-value 0. The condition that $g\le c$ guarantees that $M\in \M_1\cap\M_2$. On the other hand, for $\ell\neq i$,  
\begin{align*}
    |N\cap \cS_\ell|&=|N_1\cap \cS_\ell|+|N_2\cap \cS_\ell|-|M\cap \cS_\ell| \\
                  &\le deg_1(\cS_\ell)+deg_2(\cS_\ell)-d(s\cS_\ell)=p_\ell
\end{align*}
and 
\begin{align*}
    |N\cap \cS_i|&=|N_1\cap \cS_\ell|+|N_2\cap \cS_\ell|+1-|M\cap \cS_\ell|\\
               &\le deg_1(\cS_i)+deg_2(\cS_i)+1-d(s\cS_i)\le p_i.
\end{align*}
Similarly, $|N\cap \cT_r|\le q_r$ for each $r$. Therefore $N\in\M_1\cap \M_2$. But $M\cup N= N_1\cup N_2\cup\{v\}$, which together with other~$N_k$'s including more elements of~$V$ than $N_1\cup\cdots\cup N_C$, a contradiction. Therefore we prove that $\chi(\M_1\cap\M_2)=C$.

It remains to verify~\eqref{eq:hoffman} for every vertex subset~$U$ of~$D'$. Given $U\subseteq V(D')$, let 
\[L_U=L\cap U\quad \text{and}\quad R_U=R\cap U.\]
We claim that it is enough to verify~\eqref{eq:hoffman} when either~$U$ includes both~$s$ and~$t$, or~$U$ includes neither. Indeed, suppose $U$ includes $t$ but not $s$, then $c(\delta^{out}(U))\ge c(ts)=\infty$, in which case~\eqref{eq:hoffman} holds.
Suppose $U$ includes~$s$ but not~$t$, then the two directed edges~$s\cS_i$ and~$\cT_jt$ are not in~$\delta^{in}(U)$, and these are the only two directed edges whose~$d$-values increase in~$A'$ compared to those in~$A$. And since the $c(\delta^{out}(U))$ is non-decreasing in~$A'$ compared to that in~$A$, then the fact that the condition~\eqref{eq:hoffman} holds for $U\subseteq V(D)$ implies that the condition~\eqref{eq:hoffman} holds for such~$U\subseteq V(D')$.

First, we consider the case that $U$ includes exactly one vertex of $\{\cS_i,\cT_j\}$.

If $\cT_j\in U$ and $\cS_i\not\in U$, then in either case, $s,t\in U$ or $s,t\not\in U$, we have $s\cS_i, \cT_jt\not\in \delta^{in}(U)$. These are the only two edges whose $d$-values in~$A'$ increases compared to those in~$A$, 
therefore the condition~\eqref{eq:hoffman} holds for such~$U$ in~$D$ implies its validity in~$D'$.

If $\cS_i\in U$ and $\cT_j\not\in U$, then when $s, t\in U$, $s\cS_i\not\in\delta^{in}(U)$, $\cT_jt\in \delta^{in}(U)$ and $e_v\in \delta^{out}(U)$, therefore comparing to the condition~\eqref{eq:hoffman} for~$D$, in~$D'$, both $d(\delta^{in}(U))$ and $c(\delta^{out}(U))$ increase by one, which implies the condition~\eqref{eq:hoffman} still holds for such~$U$ in~$D'$. Similarly when $s,t\not\in U$, $s\cS_i\in\delta^{in}(U)$, $\cT_jt\not\in\delta^{in}(U)$, and $e_v\in\delta^{out}(U)$, so both $d(\delta^{in}(U))$ and $c(\delta^{out}(U))$ increase by one in~$D'$ compared to that in~$D$, therefore the condition~\eqref{eq:hoffman} for such~$U$ in~$D'$ holds. 

The final case is when~$U$ includes both~$\cS_i$ and $\cT_j$, or none of them.

If $\cS_i,\cT_j,s,t\in U$ or $\cS_i,\cT_j,s,t\not\in U$, none of the edges $s\cS_i,e_v,\cT_jt$ is between~$U$ and~$V(D')\setminus U$, then the condition~\eqref{eq:hoffman} for such~$U$ in~$D'$ is same as that for~$D$, which is true.

If $\cS_i,\cT_j\in U$ and $s,t\not\in U$. Then
\[d(\delta^{in}(U))=\sum_{\cS_\ell\in L_U}d(s\cS_\ell)  \] 
as those edges between~$L$ and~$R$ have $d$-value zero, and
\[ c(\delta^{out}(U))=\sum_{e\in E(A): e\in L_U\times (R\setminus R_U)}c(e)+\sum_{\cT_r\in R_U}c(\cT_rt).  \]
Therefore
\begin{align*}
    d(\delta^{in}(U))\le c(\delta^{out}(U))
\end{align*}
is equivalent to
\begin{equation}\label{eq:step0}
\begin{split}
      &\sum_{\cS_\ell\in L_U\setminus\{S_i\}}\Big(deg_1(\cS_\ell)+deg_2(\cS_\ell)-p_\ell\Big)+\Big(deg_1(\cS_i)+deg_2(\cS_i)+1-p_i\Big) \\
  \le & N_1(L_U, R\setminus R_U)+N_2(L_U, R\setminus R_U)+\sum_{\cT_r\in R_U}q_r.
\end{split}
\end{equation}
Because $N_1(L_U, R\setminus R_U)+N_2(L_U, R\setminus R_U)=\sum_{\cS_\ell\in L_U}\Big(deg_1(\cS_\ell)+deg_2(\cS_\ell)\Big)-N_1(L_U, R_U)-N_2(L_U,R_U)$,~\eqref{eq:step0} is equivalent to
\[ 1+N_1(L_U, R_U)+N_2(L_U,R_U)\le \sum_{\cS_\ell\in L_U}p_\ell+\sum_{\cT_r\in R_U}q_r,  \]
which is true, since $N_2(L_U,R_U)\le \sum_{\cT_r\in R_U}q_r$, and $deg_1(\cS_i)\le p_i-1$ by~\eqref{eq:NsSiNtTj} and~\eqref{eq:degkST} so that $1+N_1(L_U, R_U)\le p_i+\sum_{\cS_\ell\in L_u\setminus\{\cS_i\}}p_\ell=\sum_{\cS_\ell\in L_U}p_\ell$.

If $\cS_i,\cT_j\not\in U$ and $s,t\in U$. Then the condition~\eqref{eq:hoffman}
is equivalent to
\begin{equation}\label{eq:step1}
    \begin{split}
         &\sum_{\cT_r\in R\setminus R_U}\Big(deg_1(\cT_r)+deg_2(\cT_r)-q_r\Big)+1 \\
  \le & \sum_{\cS_\ell\in L\setminus L_U}p_\ell +|N_1(L_U, R\setminus R_U)|+|N_2(L_U, R\setminus R_U)|.
    \end{split}
\end{equation}
Since
\begin{align*}
    &|N_1(L_U, R\setminus R_U)|+|N_2(L_U, R\setminus R_U)|\\
    =& \sum_{\cT_r\in R\setminus R_U}\Big(deg_1(\cT_r)+deg_2(\cT_r) \Big)-|N_1(L\setminus L_U, R\setminus R_U)|-|N_2(L\setminus L_U,R\setminus R_U)|,
\end{align*}
\eqref{eq:step1} is equivalent to
\[  1+|N_1(L\setminus L_U, R\setminus R_U)|+|N_2(L\setminus L_U,R\setminus R_U)|\le\sum_{\cS_\ell\in L\setminus L_U}p_\ell+\sum_{\cT_r\in R\setminus R_U}q_r,  \]
which is true, since 
$|N_1(L\setminus L_U, R\setminus R_U)|\le\sum_{\cS_\ell\in L\setminus L_U}p_\ell$, and
$deg_2(\cT_j)\le q_j-1$ by~\eqref{eq:NsSiNtTj} and~\eqref{eq:degkST} implies 
$ 1+|N_2(L\setminus L_U,R\setminus R_U)|\le q_j+\sum_{\cT_r\in (R\setminus R_U)\setminus\{\cT_j\}}q_r=\sum_{\cT_r\in R\setminus R_U}q_r$,
similar as before.
\end{proof}

\section{The list chromatic number}
\begin{theorem}\label{thm:chiellchiGP}
    Let $\M_1$ and $\M_2$ be two generalized partition matroids on the same ground set. Then
    \[ \chi_\ell(\M_1\cap\M_2)=\chi(\M_1\cap \M_2).  \]
\end{theorem}

The proof of~\refT{thm:chiellchiGP} below combines the idea of Galvin's proof of \refT{thm:galvin} and a theorem of Fleiner on a matroidal version of the Gale--Shapley stable matching theorem.

To state the Fleiner's theorem,
an \emph{ordered matroid} $(\M,<)$ is a matroid $\M$ together with a linear order~$<$ on the ground set. A subset $D\subseteq V(\M)$ \emph{dominates} $v\in V(\M)$ by $\M$ if $v\in D$ or there exists $I\in \M[D]$ such that $\{v\}\cup I\not\in \M$ and $u<v$ for every $u\in I$. Given two ordered matroids $(\M_1, <_1),(\M_2, <_2)$ on the same ground set $V$, a subset $K$ of $V$ is called a \emph{kernel} if $K\in \M_1\cap \M_2$ and $K$ dominates each element $v\in V$ by~$\M_1$ or by~$\M_2$. 

\begin{theorem}[Theorem 2 in~\cite{Fleiner}]\label{thm:MSM}
    Any pair of ordered matroids on the same ground set has a kernel.
\end{theorem}

\begin{proof}[Proof of~\refT{thm:chiellchiGP}]
Let $\C=\M_1\cap \M_2$ and $V=V(\C)$.
    We shall prove that if the sizes of the list of permissible colors for each $v\in V$ are~$C:=\chi(\M_1\cap\M_2)$, then there exists an $\C$-respecting list coloring.

     Suppose~$\M_1$ has parts $\cS_1,\cdots, \cS_a$ with respective constraints $p_1,\dots,p_a$ and~$\M_2$ has parts $\cT_1,\dots,\cT_b$ with  constraints $q_1,\dots, q_b$.  
And~$N_1,\cdots, N_{C}\in\C$ are disjoint satisfying that~$\cup_{k=1}^CN_k=V$.

Now we label the elements in~$V$ distinctly by numbers $1,\cdots, |V|$ in the following way: 
for each $1\le k\le C$, the elements in $N_k$ are labeled by the numbers in 
\[\Big\{\sum_{t=1}^{k-1}|N_t|+1,\cdots,\sum_{t=1}^k|N_t|\Big\}\] 
distinctly in an arbitrary way.

We fix one of such labeling, and we define the two linear orders $<_1, <_2$ on~$V$ according to the labeling: for two distinct $u,v\in V$, $u<_1v$ if the number labeled to~$u$ is less than that of~$v$; $u<_2 v$ if the number labeled to~$u$ is greater than that of~$v$.

For ease of notation, for any~$v\in V$, we define~$i(v)$ as the index~$i$ such that $v\in \cS_i$,~$j(v)$ as the index~$j$ such that $v\in \cT_j$, and~$k(v)$ as the index~$k$ such that~$v\in N_k$.

For any induced subcomplex~$\D$ of~$\C$ and $v\in V(\D)$, we define
    \begin{equation*}
        \begin{split}
          \Gamma_{\D,1}(v)&:=\{z\in V(\D)\mid z\in \cS_{i(v)}, z<_1 v \},\\
          \Gamma_{\D,2}(v)&:=\{z\in V(\D)\mid z\in \cT_{j(v)}, z<_2 v\}. 
        \end{split}
    \end{equation*}
\begin{claim}\label{claim:kernel}
    Let $(\M_1,<_1)$, $(\M_2,<_2)$ and $N_1,\dots, N_C$  be defined as the above. Given $U \subseteq V$, let $\D:=\M_1[U]\cap\M_2[U]$. If for any $v\in U$, there exist integers $1\le t_\D(v)\le T_\D(v)$ such that
\begin{align}
|\Gamma_{\D,1}(v)|&\le (t_\D(v)-1)p_{i(v)}+p_{i(v)}-1, \label{eq:GC1}\\
|\Gamma_{\D,2}(v)|&\le (T_\D(v)-t_{\D}(v))q_{j(v)}+q_{j(v)}-1. \label{eq:GC2}
\end{align}
Then for any list~$L$ of permissible colors given to each element of~$U$ satisfying $|L_v|\ge T_{\D}(v)$ for each $v\in U$, there exists a $\D$-respecting list coloring.
\end{claim}
\begin{proof}[Proof of the claim]
    We prove by double induction on $T^*_{\D}=\max_{u\in U}T_{\D}(u)$ and on the number of elements $v\in U$ such that $T_{\D}(v)=T^*_{\D}$. 
    
    When $T^*_{\D}=1$, i.e., $T_\D(u)=t_{\D}(u)=1$ for every $u\in U$, we have $U\in \M_1\cap \M_2$: suppose not, then either there are $p_i+1$ many elements of~$U$ in some~$\cS_i$, in which case the maximum one in $<_1$ order, say $v$, has $|\Gamma_{\D,1}(v)|\ge p_{i(v)}$, contradicting to the assumption that $|\Gamma_{\D,1}(v)|\le p_{i(v)}-1$ in~\eqref{eq:GC1}; or there are $q_j+1$ many vertices of $U$ in some~$\cT_j$, in which case the maximum one in $<_2$ order, say $v$, has $|\Gamma_{\D,2}(v)|\ge q_{j(v)}$, a contradiction to~\eqref{eq:GC2}. Then $U \in \M_1[U]\cap\M_2[U]=\D$ immediately implies that for any list~$L$ of permissible colors satisfying $|L_v|=1$ for every $v\in U$, the only possible list coloring is~$\D$-respecting.

Next, we turn to $T^*_{\D}>1$. 
We take an element $v\in U$ such that $T_{\D}(v)=T^*_{\D}$ and take a color $c\in L_v$. Let $F_c:=\{u\in U\mid c\in L_u\}$. Then by~\refT{thm:MSM}, there is a kernel $K\subseteq F_c$ for $(\M_1[F_c],<_1)$ and $(\M_2[F_c],<_2)$. We define the new list~$L'$ of permissible colors for each element of $U\setminus K$ as
\begin{equation}\label{eq:newLu}
    L'_u=\begin{cases}
    L_u\setminus\{c\} &\text{ if $u\in F_c\setminus K$,}\\
    L_u &\text{ if $u\in U\setminus F_c$.}
\end{cases}
\end{equation}

It is enough to verify that for every $u\in U\setminus K$, there exist integers \[1\le t_{\D[U\setminus K]}(u)\le  T_{\D[U\setminus K]}(u)\] 
such that 
\begin{align}
|\Gamma_{\D[U\setminus K],1}(u)|&\le (t_{\D[U\setminus K]}(u)-1)p_{i(u)}+p_{i(u)}-1, \label{eq:1p} \\
|\Gamma_{\D[U\setminus K],2}(u)|&\le (T_{\D[U\setminus K]}(u)-t_{\D}(u))q_{j(u)}+q_{j(u)}-1,\label{eq:2q}
\end{align}
and
\begin{equation}\label{eq:TDUK}
     T_{\D[U\setminus K]}(u) = 
\begin{cases}
    T_{\D}(u)-1  &\text{ if $u\in F_c\setminus K$,} \\
    T_\D(u)       &\text{ if $u\in U\setminus F_c$.}
\end{cases} 
\end{equation}
Then combining~\eqref{eq:newLu} with the induction hypothesis (since $T_{\D[U\setminus K]}^*\le T_{\D}^*$ and the number of elements $u\in U\setminus K$ such that~$T_{\D[U\setminus K]}(u)=T_{\D}^*$ decreases by at least one compared to those in~$U$), for~$L'$, there exists a $\D[U\setminus K]$-respecting list coloring for elements in $U\setminus K$, which together with coloring elements in~$K$ by color~$c$ forms a $\D$-respecting list coloring for elements in~$U$. We are done.

Note that $U\setminus K=(U\setminus F_c)\cup (F_c\setminus K)$.
For each $u\in U\setminus F_c$, since 
\[|\Gamma_{\D[U\setminus K],\ell}(u)|\le |\Gamma_{\D,\ell}(u)|\]
for $\ell=1,2$,
setting $t_{\D[U\setminus K]}(u) := t_{\D}(u)$ and $T_{\D[U\setminus K]}(u) = T_{\D}(u)$ satisfies the requirements~\eqref{eq:1p},~\eqref{eq:2q}, and~\eqref{eq:TDUK}.

For each~$u\in F_c\setminus K$, by the construction~$K$ dominates~$u$ by $\M_1[F_c]$ or by~$\M_2[F_c]$. If~$K$ dominates~$u$ by $\M_1[F_c]$, then there exists $I_1\subseteq K$ and $I_1\cup\{u\}\not\in  \M_1[F_c]$ and $z<_1 u$ for every $z\in I_1$. It means there exist~$p_{i(u)}$ elements of $I_1\subseteq K$ that are in the~$\cS_{i(u)}$, and each of them is less than $u$ in $<_1$.
Thus
\begin{equation*}
\begin{split}
       |\Gamma_{\D[U\setminus K],1}(u)|&\le |\Gamma_{\D,1}(u)|-p_{i(u)}\\
    &\le (t_\D(u)-1)p_{i(u)}+p_{i(u)}-1-p_{i(u)}\\
    &\le \big((t_\D(u)-1)-1\big)p_{i(u)}+p_{i(u)}-1.
\end{split}
\end{equation*}
Thus we can set $t_{\D[U\setminus K]}(u):=t_\D(u)-1$, which is at least 1, since the above argument guarantees that $|\Gamma_{\D,1}(u)|\ge p_{i(u)}$ and then $t_\D(u)>1$.
On the other hand, setting
$T_{\D[U\setminus K]}(u)=T_{\D}(u)-1$, which is at least $t_{\D[U\setminus K]}(u)$ since $T_\D(u)\ge t_\D(u)$, we have
\begin{equation*}
\begin{split}
       |\Gamma_{\D[U\setminus K],2}(u)|\le |\Gamma_{\D,2}(u)|&\le \big(T_{\D}(u)-t_{\D}(u)\big)q_{j(u)}+q_{j(u)}-1\\
    &=\big(T_{\D[U\setminus K]}(v)-t_{\D[U\setminus K]}(v)\big)q_{j(u)}+q_{j(u)}-1.
\end{split}
\end{equation*}

If $K$ dominates $u$ by $\M_2[F_c]$, then there exists $I_2\subseteq K$ and $I_2\cup\{u\}\not\in  \M_2[F_c]$ and $z<_2 u$ for every $z\in I_2$. It means there exists~$q_{j(u)}$ elements of $I_2\subseteq K$ that are in the~$\cT_{j(u)}$, and each of them is less than $u$ in $<_2$.
Thus
\begin{equation*}
\begin{split}
       |\Gamma_{\D[U\setminus K],2}(u)|&\le |\Gamma_{\D,2}(u)|-q_{j(u)}\\
    &\le \big(T_\D(u)-t_\D(u)\big)q_{j(u)}+q_{j(u)}-1-q_{j(u)}\\
    &\le \big((T_\D(u)-1)-t_\D(u)\big)q_{j(u)}+q_{j(u)}-1.
\end{split}
\end{equation*}
Thus we set $T_{\D[U\setminus K]}(u)=T_\D(u)-1$ and $t_{\D[U\setminus K]}(u)=t_\D(u)$, and the above argument guarantees that $|\Gamma_{\D,2}(u)|\ge q_{j(u)}$ and then $T_\D(u)-t_\C(u)\ge 1$ so that $1\le t_{\D[U\setminus K]}(u)\le T_{\D[U\setminus K]}(u)$.
On the other hand,
\begin{equation*}
\begin{split}
       |\Gamma_{\D[U\setminus K],1}(u)|\le |\Gamma_{\D,1}(u)|
       &\le t_{\D}(u) p_{i(u)}+p_{i(u)}-1\\
       &=t_{\D[U\setminus K]}(u) p_{i(u)}+p_{i(u)}-1.
\end{split}
\end{equation*}
In both cases $1\le t_{\D[U\setminus K]}(u)\le T_{\D[U\setminus K]}(u)$ satisfy~\eqref{eq:1p},~\eqref{eq:2q}, and~\eqref{eq:TDUK}, which completes the proof of the claim.
\end{proof}

To complete the proof of the theorem, it is enough to verify for $U=V$ and for every $v\in V$, there exist integers $1\le t_\C(v)\le T_\C(v)\le C$ that satisfy~\eqref{eq:GC1}--\eqref{eq:GC2}.

For $v\in V$, by the setting of $<_1$, the elements of~$\Gamma_{\C,1}(v)$ have labels less than that of~$v$, which means they are in $\big(\cup_{\ell=1}^{k(v)-1}N_\ell\cap \cS_{i(v)}\big)\cup\big(N_{k(v)}\cap \cS_{i(v)}\setminus \{v\}\big)$.
Since $|N_\ell\cap \cS_i|\le p_i$ for every~$\ell$, we have
\[|\Gamma_{\C,1}(v)|\le \sum_{\ell=1}^{k(v)-1}|N_\ell\cap \cS_{i(v)}|+|N_{k(v)}\cap \cS_{i(v)}\setminus\{v\}|\le \big(k(v)-1\big)p_{i(v)}+p_{i(v)}-1. \]
And by our definition of $<_2$, elements in~$\Gamma_{\C,2}(v)$ have labels greater than that of~$v$, therefore
\[|\Gamma_{\C,2}(v)|\le \sum_{\ell=k(v)+1}^{C}|N_\ell\cap \cT_{j(v)}|+|N_{k(v)}\cap \cT_{j(v)}\setminus\{v\}|\le \big(C-k(v)\big)q_{j(v)}+q_{j(v)}-1. \]
Setting $t_\C(v)=k(v)$ and $T_\C(v)=C$, we verify the assumption~\eqref{eq:GC1} and~\eqref{eq:GC2} in~\refCl{claim:kernel} and thus prove the theorem.
\end{proof}

\bigskip{\noindent\bf Acknowledgements.} 
The author would like to thank Ron Aharoni for his helpful writing suggestions.

\small
\bibliographystyle{abbrv}

\begin{thebibliography}{10}

\bibitem{AB}
R.~Aharoni and E.~Berger.
\newblock The intersection of a matroid and a simplicial complex.
\newblock {\em Trans. Amer. Math. Soc.} {\bf 358} (2006), 4895--4917.

\bibitem{ABGK}
R.~Aharoni, E.~Berger, H.~Guo, and D.~Kotlar.
\newblock Intersection of matroids.
\newblock {\em In preparation}.

\bibitem{ABR}
R.~Aharoni, E.~Berger, and R.~Ziv.
\newblock The edge covering number of the intersection of two matroids.
\newblock {\em Discrete Math.} {\bf 312} (2012), 81--85.


\bibitem{Berczi}
K.~B\'{e}rczi, T.~Schwarcz, and Y.~Yamaguchi.
\newblock List coloring of two matroids through reduction to partition matroids.
\newblock {\em SIAM J. Discrete Math.} {\bf 35} (2021), 2192--2209.

\bibitem{Eli}
E.~Berger and H.~Guo.
\newblock Coloring the intersection of two matroids.
\newblock {\em In preparation}.

\bibitem{edmonds1965minimum}
J.~Edmonds.
\newblock Minimum partition of a matroid into independent subsets.
\newblock {\em J. Res. Nat. Bur. Standards Sect. B} {\bf 69} (1965), 67--72.

\bibitem{erdos1979choosability}
P.~Erdos, A.~L. Rubin, and H.~Taylor.
\newblock Choosability in graphs.
\newblock {\em Congr. Numer.} {\bf 26} (1979), 125--157.

\bibitem{Fleiner}
T.~Fleiner.
\newblock A matroid generalization of the stable matching polytope.
\newblock In K.~Aardal and B.~Gerards, editors, {\em Integer Programming and
  Combinatorial Optimization}, Springer, Berlin, Heidelberg (2001), 105--114.

\bibitem{GALVIN1995153}
F.~Galvin.
\newblock The list chromatic index of a bipartite multigraph.
\newblock {\em J. Combin. Theory Ser. B} {\bf 63} (1995), 153--158.

\bibitem{kiralyk}
T.~Kir{\'a}ly.
\newblock Egres open: Research forum of the egerv\'ary research group.
\newblock 2013.

\bibitem{kiraly2013open}
T.~Kir{\'a}ly.
\newblock Open questions on matroids and list colouring.
\newblock In {\em Midsummer Combinatorial Workshop} (2013), 36--38.

\bibitem{kiraly2010list}
T.~Kir{\'a}ly and J.~Pap.
\newblock On the list colouring of two matroids.
\newblock {\em EGRES Quick Proof}, (2010-01), 2010.

\bibitem{williams}
C.~S.~A. Nash-Williams.
\newblock Decomposition of finite graphs into forests.
\newblock {\em J. London Math. Soc.}  {\bf s1-39} (1964), 12--12.

\bibitem{SchrijverCO}
A.~Schrijver.
\newblock {\em Combinatorial optimization: polyhedra and efficiency},
  {\bf 24}.
\newblock Springer (2003).

\bibitem{SEYMOUR1998150}
P.~Seymour.
\newblock A note on list arboricity.
\newblock {\em J. Combin. Theory Ser. B} {\bf 72} (1998), 150--151.

\end{thebibliography}

\normalsize
\end{document}